\documentclass[11pt]{amsart}
\usepackage{amsmath}
\usepackage{amsfonts}
\usepackage{amssymb}
\usepackage{hyperref}
\hypersetup{colorlinks=true,
	linkcolor=blue,
	filecolor=magenta,
	citecolor=red,
	urlcolor=cyan,}
\newtheorem{theorem}{Theorem}
\theoremstyle{plain}

\newtheorem{lemma}{Lemma}

\numberwithin{equation}{section}
\numberwithin{lemma}{section}
\numberwithin{theorem}{section}
\numberwithin{corollary}{section}
\numberwithin{proposition}{section}
\numberwithin{remark}{section}

\begin{document}
\title{ Group of Units of Finite Group Algebras of Groups of Order 24}
	\author{Meena Sahai}
	\address{Department of Mathematics and Astronomy, University of Lucknow,Lucknow, U.P. 226007, India.}
	\email{meena\_sahai@hotmail.com}
	\author{Sheere Farhat Ansari}
	\curraddr{Department of Mathematics and Astronomy, University of Lucknow, Lucknow, U.P. 226007, India.}
	\email{sheere\_farhat@rediffmail.com}
\footnote { The financial assistance provided  to the second author in the form of a Senior Research Fellowship from University Grants Commission, INDIA is gratefully acknowledged.}

\begin{abstract}
	Let $F$ be a finite field of characteristic $p$. The structures of the unit groups of group algebras over $F$ of the three groups $D_{24}$, $S_4$ and $SL(2, \mathbb{Z}_3)$ of order $24$  are completely described in \cite{K4, SM, SM1, FM, sh1}.  In this paper, we give the unit groups of the group algebras over $F$ of the remaining groups of order $24$, namely, $C_{24}$, $C_{12} \times C_2$, $C_2^3 \times C_3$,  $C_3 \rtimes C_8$,  $C_3 \rtimes Q_8$, $D_6 \times C_4$, $C_6 \rtimes C_4$,  $C_3 \rtimes D_8$, $C_3 \times D_8$,  $C_3 \times Q_8$,   $A_4 \times C_2$ and $D_{12} \times C_2$.
	\end{abstract}

\maketitle

\section{Introduction}
Let $FG$ be the group algebra of a group $G$ over a finite field $F$ and let $U(FG)$ be the group of units in $FG$.   In this paper, we study $U(FG)$, where $G$ is a group of order $24$ and $F$ is a finite field of characteristic $p$ containing $q=p^k$ elements. We denote the Jacobson radical of $FG$ by $J(FG)$ and $V=1+ J(FG)$.

Let $D_n$, $Q_n$, and $C_n$ be the dihedral group, the quaternion group and the cyclic group of order $n$, respectively. The fifteen non-isomorphic groups of order $24$ are $C_{24}$, $C_{12} \times C_2$, $C_6 \times C_2^2$, $C_3 \rtimes C_8$, $SL(2, \mathbb{Z}_3)$, $C_3 \rtimes Q_8$, $S_4$, $C_4 \times D_6$, $D_{24}$, $C_6 \rtimes C_4$, $C_3 \rtimes D_8$, $C_3 \times D_8$, $C_3 \times Q_8$, $C_2 \times A_4$ and $D_{12} \times C_2$. The structures of $U(FS_4)$ and $U(FSL(2, \mathbb{Z}_3))$ are given  in \cite{K4, SM}, respectively.  Also, the structure of $U(FD_{24})$ for $p \leq 3$ is given in \cite{SM, FM} and for $p>3$ in  \cite{sh1}. For $p=2$, Maheshwari\cite{SM}, obtained the structure of $U(FG)$, where $G$ is a non-abelian group of order $24$.  For $p=3$,  Monaghan\cite{FM}, studied the structure of $U(FG)$ where $G$ is a non-abelian group of $24$ such that $G$ has a normal subgroup of order $3$.   Here, we completely describe the structure of $U(FG)$ in the remaining cases.

Our notations are same as in \cite{sh1, sh8}.

The number of simple components of $FG/J(FG)$ is given in \cite{F1}.
Let $m$, $\eta$ and $T$ be as in \cite{F1}. We restate here that for a $p$-regular element $g \in G$,  $\gamma_{g}$ is the sum of all the conjugates of $g$ and the cyclotomic $F$-class of $\gamma_{g}$ is 
$$S_F{(\gamma_g)}=\{\gamma_{g^t}  \mid t \in T\}.$$

\begin{lemma}\cite[Prop 1.2]{F1}\label{l1.}
	The number of simple components of $FG/J(FG)$ is equal to the number of cyclotomic $F$-classes in $G$.
\end{lemma}

\begin{lemma}\cite[Theorem 1.3]{F1}\label{l1..}
	Suppose that $Gal(F(\eta)/F)$ is cyclic. Let $w$ be the number of cyclotomic $F$-classes in $G$. If $K_1$, $K_2$, $\cdots$ , $K_w$ are the simple components of $Z(FG/J(FG))$ and $S_1$, $S_2$, $\cdots$, $S_w$ are the cyclotomic $F$- classes of $G$, then with a suitable re-ordering of indices, 
	$|S_i|=[K_i :F]$.
\end{lemma}

\section{Main Results}

Let $l \in \{3, 6, 12, 24 \}$.	If $q \equiv \pm 5, \pm 7, \pm 11$ mod $l$, then $q^2 \equiv 1$  mod $l$. So $|T| = 2$ and $|S_F(\gamma_g)| \leq 2$. Let $r$ and $s$ be the numbers of elements $g \in G$ such that $|S_F(\gamma_g)|=1$ and $|S_F(\gamma_g)|=2$, respectively. Then for $ p>3$, by Lemmas \ref{l1.} and \ref{l1..}
$$Z(FG) \cong F^r \oplus F_2^s.$$
As  $dim_F(Z(FG)) = c =$ the number of conjugacy classes in $G$, so $s = (c-r)/2$.

\begin{theorem}
	Let $F$ be a finite field of characteristic $p$ with $|F|=q=p^k$. 
	\begin{enumerate}
		\item If $p=2$, then 
		
		$U(FC_{24}) \cong
		\begin{cases}
		C_4^{6k} \times C_{8}^{3k} \times C_{2^k-1}^{3}, & \text {if  $q \equiv  1$  mod $3$;}\\
		C_4^{6k} \times C_8^{3k} \times	C_{2^k-1} \times C_{2^{2k}-1}, & \text{ if  $q \equiv -1$  mod $3$.}\\
		\end{cases}$
		\item If $p=3$, then
		
		$U(FC_{24}) \cong
		\begin{cases}
		C_3^{16k} \times C_{3^k-1}^{8}, & \text {if  $q \equiv  1$  mod $8$;}\\
		C_3^{16k} \times	C_{3^k-1}^2 \times C_{3^{2k}-1}^{3}, & \text{ if  $q \equiv -1, 3$  mod $8$;}\\
		C_3^{16k} \times	C_{3^k-1}^4 \times C_{3^{2k}-1}^{2}, & \text{ if  $q \equiv -3$  mod $8$.}\\
		\end{cases}$
		\item	If $p >3$, then 
		
		$U(FC_{24}) \cong
		\begin{cases}
		C_{p^k-1}^{24}, & \text {if $q \equiv  1$  mod $24$;}\\
		C_{p^k-1}^{2} \times C_{p^{2k}-1}^{11}, & \text{ if $q \equiv  -1, 11$  mod $24$;}\\
		C_{p^k-1}^{4} \times C_{p^{2k}-1}^{10}, & \text{ if $q \equiv 5$  mod $24$;}\\
		C_{p^k-1}^{6} \times C_{p^{2k}-1}^{9} , & \text{ if $q \equiv -5, 7$  mod $24$;}\\
		C_{p^k-1}^{8} \times C_{p^{2k}-1}^{8} , & \text{ if $q \equiv -7$  mod $24$;}\\
		C_{p^k-1}^{12} \times C_{p^{2k}-1}^{6} , & \text{ if $q \equiv -11$  mod $24$.}\\
		\end{cases}$
	\end{enumerate}
\end{theorem}

\begin{proof} Let $C_{24}=\langle x \rangle$.
	\begin{enumerate}  
		\item  If $p=2$, then $|C_{24}:C_8|\neq 0$ in $F$. So  $J(FC_{24})=\omega(C_8)$ and
		$FC_{24}/J(FC_{24}) \cong FC_3$. Thus
		$$U(FC_{24}) \cong V \times U(FC_3).$$
		
		Now,   $\alpha=\sum_{i=0}^{23}a_ix^i \in \omega(C_8)$ if and only if 
		$\sum_{j=0}^{7}a_{3j+i}=0$, $i=0$, $1$, $2$. Also, $\alpha^2=\sum_{j=0}^{11}\sum_{i=0}^{1}a_{12i+j}^2x^{2j}$, $\alpha^4=\sum_{j=0}^{5}\sum_{i=0}^{3}a_{6i+j}^4x^{4j}$ and $\alpha^8=\sum_{j=0}^{2}\sum_{i=0}^{7}a_{3i+j}^8x^{8j}=0$. 
		Since $dim_F(J(FG))=21$, $|V|=2^{21k}$.  Let $V\cong C_2^{l_1} \times C_4^{l_2} \times C_8^{l_3}$, so that $2^{21k} =2^{l_1+2l_2+3l_3}$. It is easy to see that 
		\begin{align*}
		S=& \{\alpha \in \omega(C_8) \mid  \alpha^2=0 \text{ and } \alpha = \beta^4 \text{ for some } \beta \in \omega(C_8)\},\\
		=&\{\sum_{i=0}^{2}a_{4i}x^{4i}(1+x^{12}) \mid a_{4i} \in F\}.
		\end{align*}
		Thus  $|S|=2^{3k}$ and $l_3=3k$. Also
		
		\begin{align*}
		S_1=& \{\alpha \in \omega(C_8) \mid  \alpha^2=0 \text{ and } \alpha = \beta^2 \text{ for some } \beta \in \omega(C_8)\},\\
		=&\{\sum_{i=0}^{5}a_{2i}x^{2i}(1+x^{12}) \mid a_{2i} \in F\}.
		\end{align*}
		Hence $|S_1|=2^{6k}$ and $l_2=6k$, which leads to $V \cong C_4^{6k} \times C_8^{3k}$. 
		
		For $U(FC_3)$, see  \cite[ Theorem 2.2]{sh4}.	 
		\item  If $p=3$, then $|C_{24}:C_3|=8 \neq 0$ in $F$, $J(FC_{24})=\omega(C_3)$ and
		$FC_{24}/J(FC_{24}) \cong FC_8$. Thus
		$$U(FG) \cong V \times U(FC_8).$$
		Since $dim_FJ(FC_{24})=16$ and $J(FC_{24})^3=0$,   $V\cong C_3^{16k}$. Hence 
		$$U(FC_{24}) \cong C_{3}^{16k} \times U(FC_8).$$
		The structure of $U(FC_8)$ is given in \cite[ Theorem 3.3]{sh4}.
		
		\item	 If $p>3$, then  $m=24$.
		
		If $q \equiv \pm 1$ mod $24$, then we have \cite[Lemma 2.2]{sh4}.
		
		If $q \equiv 5$ mod $24$, then $T= \{1, 5\}$ and $|S_F(\gamma_{g})|=1$ for $g=1$, $x^{\pm 6}$ and $x^{12}$. So $r=4$, $s= 10$ and  
		$$FC_{24} \cong F^4 \oplus F_{2}^{10}.$$
		
		If $q \equiv -5, 7$ mod $24$, then $|S_F(\gamma_g)|=1$ for $g=1$,  $x^{\pm 4}$, $x^{\pm 8}$, $x^{12}$. So $r = 6$, $s = 9$ and 
		$$FC_{24} \cong F^6 \oplus F_{2}^{9}.$$
		
		If $q \equiv -7$ mod $24$, then $T=\{1, 17\}$ mod $24$. Thus,
		$|S_F(\gamma_g)|=1$ for $g=1$, $x^{\pm 3}$, $x^{\pm 6}$, $x^{\pm 9}$, $x^{12}$. So $r = s = 8$ and  
		$$FC_{24} \cong F^8 \oplus F_{2}^{8}.$$
		
		If $q \equiv 11$ mod $24$, then $T=\{1, 11\}$ mod $24$. Thus,
		$|S_F(\gamma_g)|=1$ for $g=1$,  $x^{12}$. So $r = 2$, $s = 11$ and 
		$$FC_{24} \cong F^2 \oplus F_{2}^{11}.$$
		
		If $q \equiv -11$ mod $24$, then $T=\{1, 13\}$ mod $24$. Thus,
		$|S_F(\gamma_g)|=1$ for $g=1$, $x^{\pm 2}$, $x^{\pm 4}$, $x^{\pm 6}$, $x^{\pm 8}$, $x^{\pm 10}$, $x^{12}$. So $r = 12$, $s = 6$ and 
		$$FC_{24} \cong F^{12} \oplus F_{2}^{6}.$$
	\end{enumerate}
\end{proof}

\begin{theorem}
	Let $F$ be a finite field of characteristic $p$ with $|F|=q=p^k$ and let $G=C_{2} \times C_{12}$. 
	\begin{enumerate}
		\item If $p=2$, then 
		
		$U(FG) \cong
		\begin{cases}
		C_2^{9k} \times C_{4}^{6k} \times C_{2^k-1}^{3}, & \text {if  $q \equiv  1$  mod $3$;}\\
		C_2^{9k} \times C_4^{6k} \times	C_{2^k-1} \times C_{2^{2k}-1}, & \text{ if  $q \equiv -1$  mod $3$.}\\
		\end{cases}$
		\item If $p=3$, then
		
		$U(FG) \cong
		\begin{cases}
		C_3^{16k} \times C_{3^k-1}^{8}, & \text {if  $q \equiv  1$  mod $4$;}\\
		C_3^{16k} \times	C_{3^k-1}^4 \times C_{3^{2k}-1}^{2}, & \text{ if  $q \equiv -1$  mod $4$.}\\
		\end{cases}$
		\item	If $p>3$, then 
		
		$U(FG) \cong
		\begin{cases}
		C_{p^k-1}^{24}, & \text {if $q \equiv  1$  mod $12$;}\\
		C_{p^k-1}^{4} \times C_{p^{2k}-1}^{10}, & \text{ if $q \equiv  -1$  mod $12$;}\\
		C_{p^k-1}^{8} \times C_{p^{2k}-1}^{8}, & \text{ if $q \equiv 5$  mod $12$;}\\
		C_{p^k-1}^{12} \times C_{p^{2k}-1}^{6} , & \text{ if $q \equiv -5$  mod $12$.}\\
		\end{cases}$
	\end{enumerate}
\end{theorem}

\begin{proof}  Let $G=\langle x , y \mid x^{12}=y^2=(x,y)=1\rangle$.
	\begin{enumerate}
		\item Let $H= \langle x^3 \rangle \times \langle y \rangle $. 
		If $p=2$, then $|G:H|\neq 0$ in $F$, $J(FG)=\omega(H)$ and
		$FG/J(FG) \cong FC_3$. Thus
		$$U(FG) \cong V \times U(FC_3).$$
		
		Now,  let $\alpha=\sum_{j=0}^{11}\sum_{i=0}^{11}a_{12j+i}x^{i}y^j$. Then $\alpha \in \omega(H)$ if and only if 
		$\sum_{i=0}^{7}a_{3i+j}=0$ for $j=0, 1, 2$. Also, $\alpha^2=\sum_{j=0}^{5}\sum_{i=0}^{3}a_{6i+j}^2x^{2j}$ and $\alpha^4=\sum_{j=0}^{2}\sum_{i=0}^7a_{3i+j}^4x^{4j}=0$.
		Since $dim_F(J(FG))=21$, $|V|=2^{21k}$.  Let $V\cong C_2^{l_1} \times C_4^{l_2}$, so that $2^{21k} =2^{l_1+2l_2}$. Clearly
		\begin{align*}
		S=& \{\alpha \in \omega(H) \mid  \alpha^2=0 \text{ and } \alpha = \beta^2 \text{ for some } \beta \in \omega(H)\},\\
		=&\{\sum_{i=0}^{2}a_{2i}x^{2i}(1+x^{6}) \mid a_{2i} \in F\}.
		\end{align*}
		Thus  $|S|=2^{3k}$ and $l_2=3k$. So $V \cong C_2^{15k} \times C_4^{3k}$. 
		
		\item Let $K=\langle x^4 \rangle$.  If $p=3$, then $|G:K| \neq 0$ in $F$, $J(FG)=\omega(K)$ and
		$FG/J(FG) \cong F(C_2 \times C_4)$. Hence
		$$U(FG)\cong V \times U(F(C_2 \times C_4)).$$
		Since $dim_FJ(FG)=16$ and  $J(FG)^3=0$,  $V\cong C_3^{16k}$.
		
		The rest follows by \cite[ Theorem 3.4]{sh4}.
		
		\item If $p >3$, then  $m=12$. 
		
		If $q \equiv 1$ mod $12$, then $T=\{1\}$ mod $12$. Thus  $|S_F(\gamma_g)|=1$ for all $g\in G$. Therefore by Lemmas \ref{l1.} and \ref{l1..},
		$$FG \cong F^{24}.$$
		
		If $q \equiv -1$ mod $12$, then $T=\{1, 11\}$ mod $12$. Thus,
		$|S_F(\gamma_{g})|=1$ for $g=1$, $x^6$, $y$, $x^6y$. So $r = 4$, $s = 10$ and 
		$$FG \cong F^4 \oplus F_{2}^{10}.$$
		
		If $q \equiv 5$ mod $12$, then $T=\{1, 5\}$ mod $12$. Thus,
		$|S_F(\gamma_{g})|=1$ for $g=1$, $x^{\pm 3}$, $x^6$, $y$,  $x^{\pm 3}y$, $x^6y$. So $r = s = 8$ and 
		$$FG \cong F^8 \oplus F_{2}^{8}.$$
		
		If $q \equiv -5$ mod $12$, then $T=\{1, 7\}$ mod $12$. Thus,
		$|S_F(\gamma_{g})|=1$ for $g=1$, $x^{\pm 2}$, $x^{\pm 4}$,  $x^6$, $y$, $x^{\pm 2}y$, $x^{\pm 4}y$, $x^6y$. So $r = 12$, $s = 6$ and 
		$$FG \cong F^{12} \oplus F_{2}^{6}.$$
	\end{enumerate}
\end{proof}

\begin{theorem}
	Let $F$ be a finite field of characteristic $p$ with $|F|=q=p^k$ and let $G=C_{2}^3 \times C_3$.
	\begin{enumerate}
		\item If $p=2$, then 
		
		$U(FG) \cong
		\begin{cases}
		C_2^{21k} \times C_{2^k-1}^{3}, & \text {if  $q \equiv  1$  mod $3$;}\\
		C_2^{21k}  \times	C_{2^k-1} \times C_{2^{2k}-1}, & \text{ if  $q \equiv -1$  mod $3$.}\\
		\end{cases}$
		\item If $p=3$, then
		
		$U(FG) \cong C_{3}^{16k} \times C_{3^k-1}^8.$
		
		\item	If $p >3$, then	
		\begin{center}
			$U(FG) \cong
			\begin{cases}
			C_{p^k-1}^{24}, & \text {if $q \equiv  1$  mod $6$;}\\
			C_{p^k-1}^{8} \times C_{p^{2k}-1}^{8}, & \text{ if $q \equiv  -1$  mod $6$.}\\
			\end{cases}$
		\end{center}
	\end{enumerate}
\end{theorem}

\begin{proof}
	Let $C_{2}^3 \times C_{3}=\langle x , y, z, w \mid  x^{3}=1, y^2=z^2=w^2=1, xy=yx, xz=zx, xw=wx, yz=zy, yw=wy, zw=wz\rangle$.
	
	\begin{enumerate}
		\item Let $H=\langle y\rangle \times \langle z\rangle \times \langle w\rangle $.  If $p=2$, then $|G:H| \neq 0$ in $F$, $J(FG)=\omega(H)$ and
		$FG/J(FG) \cong FC_3$. Hence
		$$U(FG)\cong V \times U(FC_3).$$
		Since $dim_FJ(FG)=21$ and $\alpha^2=0$ for all $\alpha \in \omega(H)$,   $V\cong C_2^{21k}$.
		
		\item Let $K=\langle x\rangle$.  If $p=3$, then $|G:K| \neq 0$ in $F$, $J(FG)=\omega(K)$ and
		$FG/J(FG) \cong FC_2^3$. Hence
		$$U(FG)\cong V \times U(FC_2^3).$$
		Since $dim_FJ(FG)=16$ and $J(FG)^3=0$,   $V\cong C_3^{16k}$.
		
		$U(FC_2^3)$ is given in \cite[ Theorem 3.5]{sh4}.
		
		\item If $p>3$, then $m=6$. 
		If $q \equiv 1$ mod $6$, then $T=\{1\}$ mod $6$. Thus $|S_F(\gamma_g)|=1$ for all $g\in G$. Therefore by Lemmas \ref{l1.} and \ref{l1..},
		$$FG \cong F^{24}.$$
		
		If $q \equiv -1$ mod $6$, then $T=\{1, 5\}$ mod $6$. Thus,
		$|S_F(\gamma_g)|=1$ for $g=1$, $y$, $z$, $w$, $yz$, $yw$, $zw$, $yzw$. So $r= s = 8$  and 
		$$FG \cong F^8 \oplus F_{2}^{8}.$$
	\end{enumerate}
\end{proof}

Now, we discuss non-abelian groups of order $24$. Since the case $p=2$ is dealt with in \cite{SM}, we consider $p \geq 3$ only. 
\begin{table}[htp]
	\footnotesize
	\centering
	\caption{{Conjugacy classes in $C_2 \times A_4$}}	
	\begin{center}
		\begin{tabular}{|c|c|c|}
			\hline
			Representative & Elements in the class & Order of element \\
			\hline
			$[1]$ & $\{1\}$ & $1$ \\
			\hline
			$[x]$ & $\{x\}$ & $2$  \\
			\hline
			$[y]$ & $\{ y, yz, yw, yzw\}$ & $3$ \\
			\hline
			$[z]$ & $\{z, w, zw\}$ & $2$  \\
			\hline
			$[xy]$ & $\{xy, xyz, xyw, xyzw\}$ & $6$\\
			\hline
			$[xz]$ & $\{xz, xw, xzw\}$ & $2$ \\
			\hline
			$[y^{-1}]$ & $\{y^{-1}, y^{-1}z, y^{-1}w, y^{-1}zw\}$ & $3$\\
			\hline
			$[xy^{-1}]$ & $\{xy^{-1}, xy^{-1}z, xy^{-1}w, xy^{-1}zw\}$ & $6$\\
			\hline
		\end{tabular}
	\end{center}
\end{table}

\begin{theorem}
	Let $F$ be a finite field of characteristic $p$ with $|F|=q=p^k$ and let $G=C_2 \times A_4$.
	\begin{enumerate} 
		\item If $p=3$, then 
		
		$U(FG) \cong C_3^{4k} \rtimes (C_{3^k-1}^2 \times GL(3, F)^2)$.
		
		\item If $p>3$, then
		
		$U(FG) \cong
		\begin{cases}
		C_{p^k-1}^{6} \times GL(3,F)^2, & \text {if $q \equiv  1$  mod $6$;}\\
		C_{p^k-1}^2  \times C_{p^{2k}-1}^2 \times  GL(3, F)^2, & \text{ if $q \equiv -1$  mod $6$.}\\
		\end{cases}$
	\end{enumerate}
\end{theorem}
\begin{proof}  
	Let $G= \langle x, y, z, w \mid x^{2}=y^3=z^2=w^2=1,  xyx=y, xz=zx, xw=wx, zw=wz, wy=ywz, zy=yw\rangle.$
	
	\begin{enumerate}
		\item Let	 $p=3$. Clearly, $\widehat{T_3}=1+(y +y^{-1})(1+ z)(1+ w)$.
		Let $\alpha=a_0+a_1z+a_2w+a_3zw+a_4x+a_5xz+a_6xw+a_7xzw+a_8y+a_9yz+a_{10}yw+a_{11}yzw+a_{12}y^{-1}+a_{13}y^{-1}z+a_{14}y^{-1}w+a_{15}y^{-1}zw+a_{16}xy+a_{17}xyz+a_{18}xyw+a_{19}xyzw+a_{20}xy^{-1}+a_{21}xy^{-1}z+a_{22}xy^{-1}w+a_{23}xy^{-1}zw$. If $\alpha \widehat{T_3}=0$, then  for $j=0$, $1$, $2$, $3$, we have
		\begin{align*}
		a_j+\sum_{i=0}^{3}(a_{8+i}+a_{12+i})&=0;  &a_{12+j}+\sum_{i=0}^3(a_i+a_{8+i})&=0;\\
		a_{4+j}+\sum_{i=0}^{3}(a_{16+i}+a_{20+i})&=0;  &a_{16+j}+\sum_{i=0}^3(a_{4+i}+a_{20+i})&=0;\\
		a_{8+j}+\sum_{i=0}^{3}(a_{i}+a_{12+i})&=0;  &a_{20+j}+\sum_{i=0}^3(a_{4+i}+a_{16+i})&=0.\\
		\end{align*}
		
		After solving these equations, we get $a_{4i}=a_{k+4i}$ for $k= 1, 2, 3$ and $i=0, 1, 2, 3, 4, 5$. Also $a_0+a_8+a_{12}=0$ and $a_4+a_{16}+a_{20}=0$. Hence
		$Ann(\widehat{T_3})=\{((b_0+b_1x)(1-y^{-1})+(b_2+b_3x)(y-y^{-1}))(1+z+w+zw) \mid   b_i\in F\}.$
		
		Since $x$ and $(1+z)(1+w) \in Z(FG)$, $Ann(\widehat{T_3})^3=0$. Thus $Ann(\widehat{T_3}) \subseteq J(FG)$. By \cite[Lemma 2.2]{T2},  $J(FG)= Ann(\widehat{T_3})$ and $dim_F(J(FG))=4$. 
		Hence $V \cong C_3^{4k}$.
		
		As $[1]$, $[x]$, $[z]$ and $[xz]$ are the $3$-regular conjugacy classes, so $m=2$. Thus  $q \equiv 1$ mod $2$ and $S_{F}(\gamma_g)=\{\gamma_{g}\}$ for $g=1$, $x$, $z$, $xz$. Therefore by \cite[Theorem 1.3]{F1}, this yield four components in the Wedderburn decomposition of $FG/J(FG)$. Since $dim_F(FG/J(FG))=20$, we have
		$$FG/J(FG) \cong F^2 \oplus M(3,F)^2.$$
		Hence
		$$U(FG) \cong C_3^{4k} \rtimes (C_{3^k-1}^2 \times GL(3,F)^2).$$
		
		\item If $p>3$, then $m=6$.
		Since $G/G' \cong C_6$,  thus by \cite[Prop 3.6.11]{Miles},
		\begin{equation}\label{e8}
		FG \cong FC_6\oplus \big (\oplus_{i=1}^l M(n_i, F_i)\big).
		\end{equation}
		Since $dim_F(Z(FG))=8$,  $l \leq 2$.
		
		If $q\equiv 1$ mod $6$, then  $|S_F(\gamma_g)|=1$ for all $g\in G$. Therefore by Eq. (\ref{e8}), \cite[Theorem 4.1]{sh4}, Lemmas \ref{l1.} and \ref{l1..}, 
		$$FG \cong F^{6} \oplus\big( \oplus_{i=1}^2 M(n_i, F)\big).$$
		Now $\sum_{i=1}^2n_i^2=18$  gives $n_i=2$ for $i=1, 2$.
		Hence,
		$$FG \cong F^{6} \oplus M(3,F)^2.$$
		
		If $q \equiv -1$ mod $6$, then
		$|S_F(\gamma_{g})|=1$ for $g=1$, $x$, $z$, $xz$. So $r = 4$, $s = 2$ and
		$$FG \cong F^2  \oplus F_2^2 \oplus M(3,F)^2.$$
	\end{enumerate}
\end{proof}

All the remaining groups of order $24$ contain a normal subgroup of order $3$. In view of \cite{FM}, for such groups we need to consider only the semisimple case $p >3$.

\begin{table}[htp]
	\footnotesize
	\centering
	\caption{{Conjugacy classes in $C_3 \rtimes C_8$}}	
	\begin{center}
		\begin{tabular}{|c|c|c|}
			\hline
			Representative & Elements in the class & Order of element \\
			\hline
			$[1]$ & $\{1\}$ & $1$ \\
			\hline
			$[x]$ & $\{x, x^{-1}\}$ & $3$  \\
			\hline
			$[y]$ & $\{y, xy, x^{-1}y\}$ & $8$ \\
			\hline
			$[y^2]$ & $\{y^2\}$ & $4$ \\
			\hline
			$[y^3]$ & $\{y^3, xy^3, x^{-1}y^3\}$ & $8$  \\
			\hline
			$[y^4]$ & $\{y^4\}$ & $2$\\
			\hline
			$[y^{-1}]$ & $\{y^{-1}, xy^{-1}, x^{-1}y^{-1}\}$ & $8$ \\
			\hline
			$[y^{-2}]$ & $\{y^{-2}\}$ & $4$\\
			\hline
			$[y^{-3}]$ & $\{y^{-3}, xy^{-3}, x^{-1}y^{-3}\}$ & $8$\\
			\hline
			$[xy^2]$ & $\{xy^2, x^{-1}y^2\}$ & $4$ \\
			\hline
			$[xy^{-2}]$ & $\{xy^{-2}, x^{-1}y^{-2}\}$ & $4$ \\
			\hline
			$[xy^4]$ & $\{xy^4, x^{-1}y^4\}$ & $2$ \\
			\hline
		\end{tabular}
	\end{center}
\end{table}
\begin{theorem}
	Let $F$ be a finite field of characteristic $p$ with $|F|=q=p^k$ and let $G=C_3 \rtimes C_8$. If $p>3$, then  
	$$U(FG) \cong
	\begin{cases}
	C_{p^k-1}^8 \times GL(2,F)^4, & \text {if $q \equiv  1, -7$  mod $24$;}\\
	C_{p^k-1}^2 \times C_{p^{2k}-1}^{3} \times GL(2, F)^2 \times GL(2, F_2), & \text{ if $q \equiv -1, -5, 7, 11$  mod $24$;}\\
	C_{p^k-1}^4 \times C_{p^{2k}-1}^2 \times GL(2,F)^4, & \text {if $q \equiv 5, -11$  mod $24$.}\\
	\end{cases}$$
\end{theorem}
\begin{proof}  Let $G= \langle x, y \mid x^{3}=y^8=1, yxy^{-1}=x^{-1} \rangle.$
	
	$FG$ is semisimple, so all the  conjugacy  classes of $G$ are $p$-regular and $m=24$. Since $G/G' \cong C_8$,  thus by \cite[Prop 3.6.11]{Miles},
	\begin{equation}\label{e1}
	FG \cong FC_8 \oplus \big (\oplus_{i=1}^l M(n_i, F_i)\big)
	\end{equation}
	where each $F_i$ is a finite extension of $F$. Since $dim_F(Z(FG))=12$,  $l \leq 4$.
	
	If $q\equiv 1$, $-7$ mod $24$, then  $|S_F(\gamma_g)|=1$ for all $g\in G$. Therefore by Eq. (\ref{e1}), \cite[Theorem 3.3]{sh4}, Lemmas \ref{l1.} and \ref{l1..},
	$$FG \cong F^8 \oplus\big(\oplus_{i=1}^4 M(n_i, F)\big).$$
	Now $dim_F(FG)=24$, gives $n_i=2$ for $i=1, 2, 3, 4$.
	Hence,
	$$FG \cong F^{8} \oplus M(2,F)^4.$$
	
	If $q \equiv 5$, $-11$ mod $24$, then
	$|S_F(\gamma_g)|=1$ for $g=1$, $x$, $y^{\pm 2}$, $y^4$, $xy^{\pm 2}$, $xy^4$. So $r = 8$, $s = 2$ and
	$$FG \cong F^4 \oplus F_{2}^2 \oplus M(2,F)^4.$$
	
	If $q \equiv -1$, $-5$, $7$, $11$ mod $24$, then
	$|S_F(\gamma_g)|=1$ for $g=1$, $x$, $y^4$, $xy^4$. So $r = s = 4$ and
	$$FG \cong F^2 \oplus F_{2}^3 \oplus M(2,F)^2 \oplus M(2,F_2).$$
\end{proof}

\begin{table}[htp]
	\footnotesize
	\centering
	\caption{{Conjugacy classes in $C_3 \rtimes Q_8$}}	
	\begin{center}
		\begin{tabular}{|c|c|c|}
			\hline
			Representative & Elements in the class & Order of element \\
			\hline
			$[1]$ & $\{1\}$ & $1$ \\
			\hline
			$[x]$ & $\{x, x^{-1}\}$ & $12$  \\
			\hline
			$[x^2]$ & $\{x^2, x^{-2}\}$ & $6$ \\
			\hline
			$[x^3]$ & $\{x^3, x^{-3}\}$ & $4$ \\
			\hline
			$[x^4]$ & $\{x^4, x^{-4}\}$ & $3$  \\
			\hline
			$[x^5]$ & $\{x^5, x^{-5}\}$ & $12$\\
			\hline
			$[x^6]$ & $\{x^6\}$ & $2$ \\
			\hline
			$[y]$ & $\{y, x^{\pm 2}y, x^{\pm 4}y, x^6y\}$ & $4$\\
			\hline
			$[xy]$ & $\{x^{\pm 1}y, x^{\pm 3}y, x^{\pm 5}y\}$ & $4$\\
			\hline
		\end{tabular}
	\end{center}
\end{table}
\begin{theorem}
	Let $F$ be a finite field of characteristic $p$ with $|F|=q=p^k$ and let $G=C_3 \rtimes Q_8$. If $p>3$, then  
	$$U(FG) \cong
	\begin{cases}
	C_{p^k-1}^4 \times GL(2,F)^5, & \text {if $q \equiv  \pm 1$  mod $12$;}\\
	C_{p^k-1}^4  \times GL(2, F)^3 \times GL(2, F_2), & \text{ if $q \equiv \pm 5$  mod $12$.}\\
	\end{cases}$$
\end{theorem}
\begin{proof}  
	Let $G= \langle x, y \mid x^{12}=1, x^6=y^2,  yxy^{-1}=x^{-1} \rangle.$
	
	Since $p>3$ and $G/G' \cong C_2^2$,  thus by \cite[Prop 3.6.11]{Miles},	
	\begin{equation}\label{e2}
	FG \cong FC_2^2 \oplus \big (\oplus_{i=1}^l M(n_i, F_i)\big).
	\end{equation}
	Since $dim_F(Z(FG))=9$,  $l \leq 5$.
	
	If $q\equiv \pm 1$ mod $12$, then  $|S_F(\gamma_g)|=1$ for all $g\in G$. Therefore by Eq. (\ref{e2}),  \cite[Theorem 3.2]{sh4}, Lemmas \ref{l1.} and \ref{l1..}, 
	$$FG \cong F^4 \oplus\big( \oplus_{i=1}^5 M(n_i, F)\big).$$
	Now $dim_F(FG)=24$ gives  $n_i=2$ for $i=1, 2, 3, 4, 5$.
	Hence,
	$$FG \cong F^{4} \oplus M(2,F)^5.$$
	
	If $q \equiv \pm 5$ mod $12$, then
	$|S_F(\gamma_{g})|=1$ for $g=1$,  $x^2$, $x^3$, $x^4$,  $x^6$, $y$, $xy$. So $r = 7$, $s = 1$ and 
	$$FG \cong F^4  \oplus M(2,F)^3 \oplus M(2,F_2).$$
\end{proof}

\begin{table}[htp]
	\footnotesize
	\centering
	\caption{{Conjugacy classes in $C_4 \times D_6$}}	
	\begin{center}
		\begin{tabular}{|c|c|c|}
			\hline
			Representative & Elements in the class & Order of element \\
			\hline
			$[1]$ & $\{1\}$ & $1$ \\
			\hline
			$[x]$ & $\{x, x^{-1}\}$ & $3$  \\
			\hline
			$[y]$ & $\{y, xy, x^{-1}y\}$ & $2$ \\
			\hline
			$[z]$ & $\{z\}$ & $4$ \\
			\hline
			$[z^2]$ & $\{z^2\}$ & $2$  \\
			\hline
			$[z^{-1}]$ & $\{z^{-1}\}$ & $4$\\
			\hline
			$[xz]$ & $\{xz, x^{-1}z\}$ & $12$ \\
			\hline
			$[xz^2]$ & $\{xz^2, x^{-1}z^2\}$ & $6$\\
			\hline
			$[xz^{-1}]$ & $\{xz^{-1}, x^{-1}z^{-1}\}$ & $12$\\
			\hline
			$[yz]$ & $\{yz, xyz, x^{-1}yz\}$ & $4$\\
			\hline
			$[yz^2]$ & $\{yz^2, xyz^2, x^{-1}yz^2\}$ & $2$\\
			\hline
			$[yz^{-1}]$ & $\{yz^{-1}, xyz^{-1}, x^{-1}yz^{-1}\}$ & $4$\\
			\hline
		\end{tabular}
	\end{center}
\end{table}
\begin{theorem}
	Let $F$ be a finite field of characteristic $p$ with $|F|=q=p^k$ and let $G=C_4 \times D_6$. If $p>3$, then  
	$$U(FG) \cong
	\begin{cases}
	C_{p^k-1}^8 \times GL(2,F)^4, & \text {if $q \equiv  1, 5$  mod $12$;}\\
	C_{p^k-1}^4  \times C_{p^{2k}-1}^2 \times  GL(2, F)^2 \times GL(2, F_2), & \text{ if $q \equiv -1, -5$  mod $12$.}\\
	\end{cases}$$
\end{theorem}
\begin{proof}  
	Let $G= \langle x, y, z \mid x^{3}=y^2=z^4=xyxy=1, xz=zx, yz=zy \rangle.$
	
	Since $G/G' \cong (C_2 \times C_4)$,  thus by \cite[Prop 3.6.11]{Miles},
	\begin{equation}\label{e3}
	FG \cong F(C_2 \times C_4)\oplus \big (\oplus_{i=1}^l M(n_i, F_i)\big).
	\end{equation}
	Since $dim_F(Z(FG))=12$,  $l \leq 4$.
	
	If $q\equiv 1$, $5$ mod $12$, then  $|S_F(\gamma_g)|=1$ for all $g\in G$. Therefore by Eq. (\ref{e3}),  \cite[Theorem 3.4]{sh4}, Lemmas \ref{l1.} and \ref{l1..}, 
	$$FG \cong F^8 \oplus\big( \oplus_{i=1}^4 M(n_i, F)\big).$$
	Now $\sum_{i=1}^{4}n_i^2=16$  gives $n_i=2$ for $i=1, 2, 3, 4$.  
	Hence,
	$$FG \cong F^{8} \oplus M(2,F)^4.$$
	
	If $q \equiv -1$, $-5$ mod $12$, then
	$|S_F(\gamma_{g})|=1$ for $g=1$, $x$, $y$, $z^2$, $xz^2$, $yz^2$. So $r = 6$, $s = 3$ and 
	$$FG \cong F^4  \oplus F_2^2 \oplus M(2,F)^2 \oplus M(2,F_2).$$
\end{proof}

\begin{table}[htp]
	\footnotesize
	\centering
	\caption{{Conjugacy classes in $C_6 \rtimes C_4$}}	
	\begin{center}
		\begin{tabular}{|c|c|c|}
			\hline
			Representative & Elements in the class & Order of element \\
			\hline
			$[1]$ & $\{1\}$ & $1$ \\
			\hline
			$[x]$ & $\{x, xy^2, xy^{-2}\}$ & $4$  \\
			\hline
			$[x^2]$ & $\{x^2\}$ & $2$ \\
			\hline
			$[x^{-1}]$ & $\{ x^{-1}, x^{-1}y^{-2}, x^{-1}y^2\}$ & $4$ \\
			\hline
			$[y]$ & $\{y, y^{-1}\}$ & $6$  \\
			\hline
			$[y^2]$ & $\{y^2, y^{-2}\}$ & $3$\\
			\hline
			$[y^3]$ & $\{y^3\}$ & $2$ \\
			\hline
			$[xy]$ & $\{xy, xy^{-1}, xy^{3}\}$ & $4$\\
			\hline
			$[x^2y]$ & $\{x^2y, x^2y^{-1}\}$ & $6$\\
			\hline
			$[x^{-1}y]$ & $\{x^{-1}y, x^{-1}y^{-1}, x^{-1}y^3\}$ & $4$\\
			\hline
			$[x^2y^2]$ & $\{x^2y^2, x^2y^{-2}\}$ & $4$\\
			\hline
			$[x^2y^3]$ & $\{x^2y^3\}$ & $2$\\
			\hline
		\end{tabular}
	\end{center}
\end{table}
\begin{theorem}
	Let $F$ be a finite field of characteristic $p$ with $|F|=q=p^k$ and let $G=C_6 \rtimes C_4$. If $p>3$, then  
	$$U(FG) \cong
	\begin{cases}
	C_{p^k-1}^8 \times GL(2,F)^4, & \text {if $q \equiv  1, 5$  mod $12$;}\\
	C_{p^k-1}^4  \times C_{p^{2k}-1}^2 \times  GL(2, F)^4, & \text{ if $q \equiv -1, -5$  mod $12$.}\\
	\end{cases}$$
\end{theorem}
\begin{proof}  
	Let $G= \langle x, y \mid x^{4}=y^6=1,  yxy=x\rangle.$
	
	Since $G/G' \cong (C_2 \times C_4)$,  thus by \cite[Prop 3.6.11]{Miles},
	\begin{equation}\label{e4}
	FG \cong F(C_2 \times C_4)\oplus \big (\oplus_{i=1}^l M(n_i, F_i)\big).
	\end{equation}
	Since $dim_F(Z(FG))=12$,  $l \leq 4$.
	
	If $q\equiv 1$, $5$ mod $12$, then  $|S_F(\gamma_g)|=1$ for all $g\in G$. Therefore by Eq. (\ref{e4}), \cite[Theorem 3.4]{sh4}, Lemmas \ref{l1.} and \ref{l1..}, 
	$$FG \cong F^8 \oplus\big( \oplus_{i=1}^4 M(n_i, F)\big).$$
	Now $\sum_{i=1}^{4}n_i^2=16$  gives $n_i=2$ for $i=1, 2, 3, 4$.  
	Hence,
	$$FG \cong F^{8} \oplus M(2,F)^4.$$
	
	If $q \equiv -1$, $-5$ mod $12$, then
	$|S_F(\gamma_{g})|=1$ for $g=1$,  $x^2$, $y$, $y^2$, $y^3$, $x^2y$, $x^2y^2$, $x^2y^3$. So $r = 8$, $s = 2$ and
	$$FG \cong F^4  \oplus F_2^2 \oplus M(2,F)^4.$$
\end{proof}

\begin{table}[htp]
	\footnotesize
	\centering
	\caption{{Conjugacy classes in $C_3 \rtimes D_8$}}	
	\begin{center}
		\begin{tabular}{|c|c|c|}
			\hline
			Representative & Elements in the class & Order of element \\
			\hline
			$[1]$ & $\{1\}$ & $1$ \\
			\hline
			$[x]$ & $\{x^{\pm 1}\}$ & $3$  \\
			\hline
			$[y]$ & $\{y^{\pm 1}, xy^{\pm 1}, x^{-1}y^{\pm 1}\}$ & $4$ \\
			\hline
			$[z]$ & $\{ z, y^2z\}$ & $2$ \\
			\hline
			$[y^2]$ & $\{y^2\}$ & $2$  \\
			\hline
			$[yz]$ & $\{y^{\pm 1}z, xy^{\pm 1}z, x^{-1}y^{\pm 1}z\}$ & $2$\\
			\hline
			$[xz]$ & $\{xz, x^{-1}y^2z\}$ & $6$ \\
			\hline
			$[x^{-1}z]$ & $\{x^{-1}z, xy^2z\}$ & $6$\\
			\hline
			$[xy^2]$ & $\{x^{\pm 1}y^2\}$ & $6$\\
			\hline
		\end{tabular}
	\end{center}
\end{table}
\begin{theorem}
	Let $F$ be a finite field of characteristic $p$ with $|F|=q=p^k$ and let $G=C_3 \rtimes D_8$. If $p>3$, then  
	$$U(FG) \cong
	\begin{cases}
	C_{p^k-1}^4 \times GL(2,F)^5, & \text {if $q \equiv  1, 5$  mod $12$;}\\
	C_{p^k-1}^4  \times   GL(2, F)^3 \times GL(2, F_2), & \text{ if $q \equiv -1, -5$  mod $12$.}\\
	\end{cases}$$
\end{theorem}
\begin{proof}  
	Let $G= \langle x, y, z \mid x^{3}=y^4=z^2= yzyz=1,  xyx=y,xz=zx\rangle.$
	
	Since $G/G' \cong C_2^2$,  thus by \cite[Prop 3.6.11]{Miles},
	\begin{equation}\label{e9}
	FG \cong  FC_2^2\oplus \big (\oplus_{i=1}^l M(n_i, F_i)\big).
	\end{equation}
	Since $dim_F(Z(FG))=9$,  $l \leq 5$.
	
	If $q\equiv 1$, $5$ mod $12$, then  $|S_F(\gamma_g)|=1$ for all $g\in G$. Therefore by Eq. (\ref{e9}), \cite[Theorem 3.2]{sh4}, Lemmas \ref{l1.} and \ref{l1..}, 
	$$FG \cong F^{4} \oplus\big( \oplus_{i=1}^5 M(n_i, F)\big).$$
	Now $\sum_{i=1}^{5}n_i^2=20$  gives $n_i=2$ for $i=1, 2, 3, 4, 5$.
	Hence,
	$$FG \cong F^{4} \oplus M(2,F)^5.$$
	
	If $q \equiv -1$, $-5$ mod $12$, then
	$|S_F(\gamma_{g})|=1$ for $g=1$, $x$, $y$, $z$,  $y^2$, $yz$, $xy^2$. So $r = 7$, $s = 1$ and 
	$$FG \cong F^4 \oplus M(2,F)^3 \oplus M(2, F_2).$$
\end{proof}

\begin{table}[htp]
	\footnotesize
	\centering
	\caption{{Conjugacy classes in $C_3 \times D_8$}}	
	\begin{center}
		\begin{tabular}{|c|c|c|}
			\hline
			Representative & Elements in the class & Order of element \\
			\hline
			$[1]$ & $\{1\}$ & $1$ \\
			\hline
			$[x]$ & $\{x\}$ & $3$  \\
			\hline
			$[x^{-1}]$ & $\{x^{-1}\}$ & $3$ \\
			\hline
			$[y]$ & $\{ y, y^{-1}\}$ & $4$ \\
			\hline
			$[y^2]$ & $\{y^{2}\}$ & $2$  \\
			\hline
			$[z]$ & $\{z, y^2z\}$ & $2$\\
			\hline
			$[xy]$ & $\{xy, xy^{-1}\}$ & $12$ \\
			\hline
			$[x^{-1}y]$ & $\{x^{-1}y, x^{-1}y^{-1}\}$ & $12$\\
			\hline
			$[xy^2]$ & $\{xy^2\}$ & $6$\\
			\hline
			$[x^{-1}y^2]$ & $\{x^{-1}y^2\}$ & $6$\\
			\hline
			$[xz]$ & $\{xz, xy^2z\}$ & $6$\\
			\hline
			$[x^{-1}z]$ & $\{x^{-1}z, x^{-1}y^2z\}$ & $6$\\
			\hline
			$[yz]$ & $\{yz, y^{-1}z\}$ & $2$\\
			\hline
			$[xyz]$ & $\{xyz, xy^{-1}z\}$ & $6$\\
			\hline
			$[x^{-1}yz]$ & $\{x^{-1}yz, x^{-1}y^{-1}z\}$ & $6$\\
			\hline
		\end{tabular}
	\end{center}
\end{table}
\begin{theorem}
	Let $F$ be a finite field of characteristic $p$ with $|F|=q=p^k$ and let $G=C_3 \times D_8$. If $p>3$, then  
	$$U(FG) \cong
	\begin{cases}
	C_{p^k-1}^{12} \times GL(2,F)^3, & \text {if $q \equiv  1, -5$  mod $12$;}\\
	C_{p^k-1}^4  \times C_{p^{2k}-1}^4 \times  GL(2, F) \times GL(2, F_2), & \text{ if $q \equiv -1, 5$  mod $12$.}\\
	\end{cases}$$
\end{theorem}
\begin{proof}  
	Let $G= \langle x, y \mid x^{3}=y^4=z^2=yzyz=1, xy=yx, xz=zx\rangle.$
	
	Since $G/G' \cong (C_2 \times C_6)$,  thus by \cite[Prop 3.6.11]{Miles},
	\begin{equation}\label{e6}
	FG \cong F(C_2 \times C_6)\oplus \big (\oplus_{i=1}^l M(n_i, F_i)\big).
	\end{equation}
	Since $dim_F(Z(FG))=15$,  $l \leq 3$.
	
	If $q\equiv 1$, $-5$ mod $12$, then  $|S_F(\gamma_g)|=1$ for all $g\in G$. Therefore by Eq. (\ref{e6}), \cite[Theorem 3.5]{T1}, Lemmas \ref{l1.} and \ref{l1..}, 
	$$FG \cong F^{12} \oplus\big( \oplus_{i=1}^3 M(n_i, F)\big).$$
	Now $\sum_{i=1}^{3}n_i^2=12$  gives $n_i=2$ for $i=1, 2, 3$.
	Hence,
	$$FG \cong F^{12} \oplus M(2,F)^3.$$
	
	If $q \equiv -1$, $5$ mod $12$, then
	$|S_F(\gamma_{g})|=1$ for $g=1$, $y$, $y^2$, $z$,  $yz$. So $r = s = 5$ and
	$$FG \cong F^4  \oplus F_2^4 \oplus M(2,F) \oplus M(2, F_2).$$
\end{proof}

\begin{table}[htp]
	\footnotesize
	\centering
	\caption{{Conjugacy classes in $C_3 \times Q_8$}}	
	\begin{center}
		\begin{tabular}{|c|c|c|}
			\hline
			Representative & Elements in the class & Order of element \\
			\hline
			$[1]$ & $\{1\}$ & $1$ \\
			\hline
			$[x]$ & $\{x, x^{-1}\}$ & $4$  \\
			\hline
			$[x^2]$ & $\{x^{2}\}$ & $2$ \\
			\hline
			$[y]$ & $\{ y, y^{-1}\}$ & $4$ \\
			\hline
			$[xy]$ & $\{xy, x^{-1}y\}$ & $4$  \\
			\hline
			$[z]$ & $\{z\}$ & $3$\\
			\hline
			$[z^{-1}]$ & $\{z^{-1}\}$ & $3$ \\
			\hline
			$[xz]$ & $\{xz, x^{-1}z\}$ & $12$\\
			\hline
			$[x^2z]$ & $\{x^2z\}$ & $6$\\
			\hline
			$[xz^{-1}]$ & $\{xz^{-1}, x^{-1}z^{-1}\}$ & $12$\\
			\hline
			$[x^2z^{-1}]$ & $\{x^2z^{-1}\}$ & $6$\\
			\hline
			$[yz]$ & $\{yz, y^{-1}z\}$ & $12$\\
			\hline
			$[xyz]$ & $\{xyz, x^{-1}yz\}$ & $12$\\
			\hline
			$[xyz^{-1}]$ & $\{xyz^{-1}, x^{-1}yz^{-1}\}$ & $12$\\
			\hline
			$[yz^{-1}]$ & $\{yz^{-1}, y^{-1}z^{-1}\}$ & $12$\\
			\hline
		\end{tabular}
	\end{center}
\end{table}
\begin{theorem}
	Let $F$ be a finite field of characteristic $p$ with $|F|=q=p^k$ and let $G=C_3 \times Q_8$. If $p >3$, then  
	$$U(FG) \cong
	\begin{cases}
	C_{p^k-1}^{12} \times GL(2,F)^3, & \text {if $q \equiv  1, -5$  mod $12$;}\\
	C_{p^k-1}^4  \times C_{p^{2k}-1}^4 \times  GL(2, F) \times GL(2, F_2), & \text{ if $q \equiv -1, 5$  mod $12$.}\\
	\end{cases}$$
\end{theorem}
\begin{proof}  
	Let $G= \langle x, y, z \mid x^{4}=z^3=1,  x^2=y^2, yxy^{-1}=x^{-1}, xz=zx, yz=zy\rangle.$
	
	Since $G/G' \cong (C_2 \times C_6)$,  thus by \cite[Prop 3.6.11]{Miles},
	\begin{equation}\label{e7}
	FG \cong F(C_2 \times C_6)\oplus \big (\oplus_{i=1}^l M(n_i, F_i)\big).
	\end{equation}
	Since $dim_F(Z(FG))=15$,  $l \leq 3$.
	
	If $q\equiv 1$, $-5$ mod $12$, then  $|S_F(\gamma_g)|=1$ for all $g\in G$. Therefore by Eq. (\ref{e7}), \cite[Theorem 3.5]{T1}, Lemmas \ref{l1.} and \ref{l1..}, 
	$$FG \cong F^{12} \oplus\big( \oplus_{i=1}^3 M(n_i, F)\big).$$
	Now $\sum_{i=0}^{3}n_i^2=12$ gives $n_i=2$ for $i=1, 2, 3$.
	Hence,
	$$FG \cong F^{12} \oplus M(2,F)^3.$$
	
	If $q \equiv -1, 5$ mod $12$, then
	$|S_F(\gamma_{g})|=1$ for $g=1$, $x$, $x^2$, $y$, $xy$. So $ r = s = 5$ and 
	$$FG \cong F^4  \oplus F_2^4 \oplus M(2,F) \oplus M(2, F_2).$$
\end{proof}

\begin{table}[htp]
	\footnotesize
	\centering
	\caption{{Conjugacy classes in $D_{12} \times C_2$}}	
	\begin{center}
		\begin{tabular}{|c|c|c|}
			\hline
			Representative & Elements in the class & Order of element \\
			\hline
			$[1]$ & $\{1\}$ & $1$ \\
			\hline
			$[x]$ & $\{x, x^{-1}\}$ & $6$  \\
			\hline
			$[x^2]$ & $\{x^{2}, x^{-2}\}$ & $3$ \\
			\hline
			$[x^3]$ & $\{ x^3\}$ & $2$ \\
			\hline
			$[y]$ & $\{y, x^2y, x^{-2}y, \}$ & $2$  \\
			\hline
			$[xy]$ & $\{xy, x^{-1}y, x^3y\}$ & $2$\\
			\hline
			$[z]$ & $\{z\}$ & $2$ \\
			\hline
			$[xz]$ & $\{xz, x^{-1}z\}$ & $6$\\
			\hline
			$[x^2z]$ & $\{x^2z, x^{-2}z\}$ & $6$\\
			\hline
			$[x^3z]$ & $\{x^3z\}$ & $2$\\
			\hline
			$[yz]$ & $\{yz, x^2yz, x^{-2}yz\}$ & $2$\\
			\hline
			$[xyz]$ & $\{xyz, x^{-1}yz, x^3yz\}$ & $2$\\
			\hline
		\end{tabular}
	\end{center}
\end{table}
\begin{theorem}
	Let $F$ be a finite field of characteristic $p$ with $|F|=q=p^k$ and let $G=D_{12} \times C_2$. If $p>3$, then  
	$$U(FG) \cong	C_{p^k-1}^8 \times GL(2,F)^4,  \text{if $q \equiv  \pm 1$  mod $6$}.$$
\end{theorem}
\begin{proof}
	Let $G= \langle x, y \mid x^{6}=y^2=z^2= xyxy = 1, xz=zx, yz=zy\rangle.$
	
	Since $G/G' \cong C_2^3$,  thus by \cite[Prop 3.6.11]{Miles},
	\begin{equation}\label{9}
	FG \cong FC_2^3 \oplus \big (\oplus_{i=1}^l M(n_i, F_i)\big)
	\end{equation}
	Since $dim_F(Z(FG))=12$,  $l \leq 4$.
	
	If $q\equiv \pm 1$ mod $6$, then $|S_F(\gamma_g)|=1$ for all $g\in G$. Therefore by Eq. (\ref{9}), \cite[Theorem 3.5]{sh4}, Lemmas \ref{l1.} and \ref{l1..}, 
	$$FG \cong F^8 \oplus \big(\oplus_{i=1}^4 M(n_i,F)\big).$$
	Now $\sum_{i=1}^{4}n_i^2=16$
	gives $n_i=2$ for $i=1, 2, 3, 4$. Hence,
	$$FG \cong F^{8} \oplus M(2,F)^4.$$
\end{proof}

{\bf Acknowledgement.} The financial assistance provided  to the second author in the form of a Senior Research Fellowship from University Grants Commission, INDIA is gratefully acknowledged.

\end{document}